\documentclass[a4paper]{amsart} 

\usepackage{amsfonts}
\usepackage{amsmath}  
\usepackage{amscd}
\usepackage{graphicx}
\usepackage{psfrag}

\usepackage{amssymb}
\usepackage{latexsym}
\usepackage{framed, color}

\newtheorem{thm}{\sc Theorem}[section]

\newtheorem{prop}[thm]{\sc Proposition}
\theoremstyle{definition}
\newtheorem{dfn}[thm]{\sc Definition}

\theoremstyle{remark}
\newtheorem{exam}[thm]{\sc Example}
\newtheorem{rmk}[thm]{\sc Remark}
\newtheorem{prob}[thm]{\sc Problem}
\makeatletter
 
 \@addtoreset{equation}{section}
\makeatother

\newcommand{\R}{\mathbf{R}}

\newcommand{\Z}{\mathbf{Z}}

\newcommand{\Int}{\mathop{\mathrm{Int}}\nolimits}

\newcommand{\rank}{\mathop{\mathrm{rank}}\nolimits}

\renewcommand{\setminus}{\smallsetminus}
\def\spmapright#1{\smash{%
 \mathop{\hbox to 1.3cm{\rightarrowfill}}
  \limits^{#1}}}

\title[Elimination of definite fold II]{Elimination of definite fold II}
\dedicatory{Dedicated to Professor Takashi
Nishimura on the occasion of his 60th birthday}
\author{Osamu Saeki} 
\address{Institute of Mathematics for Industry,
Kyushu University,
Motooka 744, Nishi-ku, Fukuoka 819-0395, Japan}

\email{saeki@imi.kyushu-u.ac.jp}
\date{\today}

\begin{document}
\begin{abstract}
In this paper, we first give a new simple proof to the
elimination theorem of definite fold by homotopy for generic smooth
maps of manifolds of dimension strictly greater than $2$
into the $2$--sphere or into the real projective plane.
Our new proof has the
advantage that it is not only constructive, but is 
also algorithmic: the procedures
enable us to construct various
explicit examples. We also study simple stable maps
of $3$--manifolds into the $2$--sphere without definite fold.
Furthermore, we prove the
non-existence of singular Legendre fibrations
on $3$--manifolds, answering negatively to a question posed
in our previous paper.
\end{abstract}

\maketitle 

\section{Introduction}\label{section1}

This is a continuation of our previous paper \cite{Saeki06} in which
we proved that for an arbitrary $C^\infty$ stable (or excellent)
map of a closed
manifold of dimension $n > 2$ into the $2$--sphere $S^2$ 
or into the real projective plane $\R P^2$,
we can eliminate definite fold points by homotopy.

The present paper has basically four purposes. The first one
is to give a new simple proof to the elimination theorem
of definite fold in which we modify a given map in
its homotopy class (see \S\ref{section2}). Recall that in \cite{Saeki06}, we used
surgery operations, which modified the source manifold
of the given map to another manifold. This made an explicit
construction very hard to realize. On the contrary, our
new proof presented in this paper modifies the original
map only by homotopy: furthermore, it is not only constructive,
but is also algorithmic, which
is an important ingredient for the simplification process
for maps on $4$--manifolds
proposed in \cite{BS} (see also \cite{BS2}).

The second purpose of the present paper is to give several explicit examples
of $C^\infty$ stable maps without definite fold points (see \S\ref{section3}).
We explicitly construct such a map in each element
of the homotopy groups $\pi_3(S^2) \cong \Z$ and $\pi_4(S^2) \cong \Z_2$.
We also construct explicit $C^\infty$ stable maps on homotopy $n$--spheres,
$n \geq 5$, without definite fold points. Note that these examples 
are now easy consequences of our new constructive proof.

The third purpose of the present paper is to study
simple stable maps of $3$--manifolds into $S^2$ without
definite fold (see \S\ref{section4}). 
We show that a closed orientable $3$--manifold admits
a $C^\infty$ stable map without definite fold into $S^2$ with embedded
indefinite fold image if and only if it is a graph manifold, by
combining our techniques with those developed in \cite{Saeki4}.
We also study the construction of such a $C^\infty$ stable
map in a given homotopy class.

The final purpose of the present paper is to show that
there exist no singular Legendre fibrations on
an arbitrary orientable $3$--manifold, answering negatively
to a problem posed in \cite[\S3]{Saeki06} (see \S\ref{section5}).
Recall that it had been known that such a singular Legendre fibration
must be an excellent map without definite fold.
We will show that even locally, there exists no
contact structure which makes fibers near
an indefinite fold point all Legendrian.
The result is somewhat surprising, but the proof
is elementary.

Throughout the paper, manifolds and maps
are differentiable of class $C^\infty$
unless otherwise indicated.
A stable map will always mean a $C^\infty$ stable map
(for stable maps, the reader is referred to \cite{GG, Levine1}).

\section{A new proof}\label{section2}

Let $M$ be a closed $n$--dimensional manifold
with $n \geq 2$ and $N$ a surface.
For a smooth map $f : M \to N$,
we denote the set of singular points by
$$S(f) = \{x \in M\,|\, \rank{df_x} < 2 \}.$$

\begin{dfn}\label{dfn1}
(1) A singular point $x \in S(f)$ of $f$ is called a
\emph{fold point} if there exist
local coordinates $(x_1, x_2, \ldots, x_n)$ around
$x$ and $(y_1, y_2)$ around $f(x)$ such that
$f$ has the form
$$y_1 \circ f = x_1, \quad 
y_2 \circ f =
-x_2^2 - \cdots - x_{1+\lambda}^2 + x_{2+\lambda}^2
+ \cdots + x_n^2
$$
for some $\lambda$ with $0 \leq \lambda \leq n-1$.
The integer $\lambda$ is called
the \emph{index} of the fold point with respect
to the $y_2$--direction: this means that $x$ corresponds
to a non-degenerate critical point of index $\lambda$
for the function $y_2 \circ f$ restricted to
the submanifold $y_1 = 0$.
We say that $x$ is a \emph{definite fold
point} if $\lambda = 0$ or $n-1$; otherwise, it is an
\emph{indefinite fold point}.
We denote by $S_0(f)$ (resp.\ $S_1(f)$)
the set of definite (resp.\ indefinite)
fold points of $f$.

(2) A singular point $x \in S(f)$ of $f$ is called a
\emph{cusp point} 
if there exist
local coordinates $(x_1, x_2, \ldots, x_n)$ around
$x$ and $(y_1, y_2)$ around $f(x)$ such that
$f$ has the form
$$y_1 \circ f = x_1, \quad
y_2 \circ f =
x_1x_2+ x_2^3 \pm x_3^2 \pm \cdots \pm x_n^2.
$$

(3) A smooth map $f : M \to N$ is called an \emph{excellent
map} if $S(f)$ consists only of fold and cusp points.
\end{dfn}

It is known that every smooth map $f : M \to N$
can be approximated by (and hence is homotopic to)
an excellent map (or a stable map).
Note that a smooth map $f : M \to N$ is stable if and only if
it is excellent and $f|_{S(f)}$ satisfies certain normal crossing
conditions (for details, see \cite{GG}, for example).

The following theorem was announced and proved
in \cite[Theorem~2.6]{Saeki06}. Here we give a much
simpler proof without performing surgeries of manifolds.

\begin{thm}\label{thm}
Let $M$ be a closed manifold of dimension $n
> 2$ and $N$ be the $2$--sphere $S^2$ or the
real projective plane $\R P^2$.
Then every continuous map $f : M
\to N$ is homotopic to a $C^\infty$ stable
map without definite fold points.
\end{thm}

\begin{proof}
We may assume that $M$ is connected and that
$f$ is stable.
By the swallow-tail move (or flip) as described in \cite[Lemma~3.3]{Saeki3}
applied to $S_0(f)$,
we see that $f$ is homotopic to a stable map
$f_1$ such that each connected component of $S_0(f_1)$ is
an open arc. Note that the two cusps appearing
at the ends of each component constitute a matching pair
in the sense of \cite{Levine0, Saeki3}.
Then, since $M$ is connected, we can connect the components
by appropriate joining curves, and by the beak-to-beak move 
(or cusp merge) as described in \cite{Levine0} or
\cite[Lemma~3.7]{Saeki3}, we see that $f_1$ is homotopic
to a stable map $f_2$ such that
$S_0(f_2)$ is a connected open arc. Then, again by using the same
move with a joining curve connecting the pair of cusps
in its ends, we see that $f_2$ is homotopic to
a stable map $f_3$ such that $S_0(f_3)$ is a circle.
Note that, in this last move, we can choose the joining
curve parallel to $S_0(f_2)$ so that $f_3|_{S_0(f_3)}$
is null-homotopic.
Furthermore, by applying the move described in
Remark~\ref{rmk:tau} below if necessary, we may assume that
the immersion
$f_3|_{S_0(f_3)}$ is regularly homotopic to an embedding.

Let $c$ be an immersed curve in the surface $N$ with normal
crossings. Let us consider
a $2$--disk $D$ with two corner points embedded in $N$
such that 
$c \cap \partial D$ consists of an embedded arc, say $\alpha$, in $c$,
which coincides with one of the two smooth boundary curves
of $D$, possibly together with some transverse intersection points
of $c$ and $\beta$, where $\beta$ is the closure of $\partial D \setminus \alpha$.
We assume that the two corner points of $\partial D$
are not double points of $c$ and that $\beta$
does not contain a double point of $c$.
Then the \emph{disk move} with respect to $D$ transforms $c$ to 
the closure of
$$(c \setminus \alpha) \cup \beta$$
with the corners smoothed (see Figure~\ref{fig11}).
It is clear that the resulting immersed curve $c'$ with normal crossings
is regularly homotopic to $c$.
Furthermore, it is not difficult to show that any
regular homotopy is realized by a finite iteration of
disk moves. (In fact, Reidemeister type moves II and III are
so realized, and also any isotopy is also so realized.)

\begin{figure}[htbp]
\begin{center}
\psfrag{a}{$\alpha$}
\psfrag{b}{$\beta$}
\psfrag{c}{$c$}
\psfrag{c1}{$c'$}
\psfrag{D}{$D$}
\includegraphics[width=0.9\linewidth,height=0.4\textheight,
keepaspectratio]{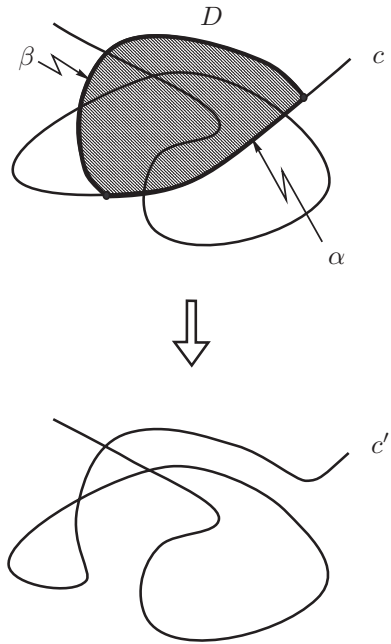}
\end{center}
\caption{Disk move. The curve $\alpha$ is replaced by the
curve $\beta$ up to slight smoothing.}
\label{fig11}
\end{figure}

Now, let us consider our situation: $c = f_3(S_0(f_3))$ 
is the image of the definite
fold curve immersed in $S^2$ (or $\R P^2$). If a small collar neighborhood
of $\alpha$ in $D$ is not contained in the image of a
small tubular neighborhood of the definite fold in $M$, then it is easy
to realize the disk move by a homotopy of the excellent map, even
under the presence of indefinite fold image.
Otherwise, we can use the method described in 
\cite[Case~2, p.~375]{Saeki06} in order to realize the disk move, at the
cost of creating two small embedded circles of indefinite fold image
(see the procedures described by Figures~6--8 of \cite{Saeki06}).
Repeating this procedure finitely many times, we get a stable map $f_4$ homotopic to $f_3$
such that $f_4|_{S_0(f_4)}$ is an embedding.

Moreover, we may assume that $f_4(S_0(f_4))$ bounds
a $2$--disk $\Delta$ in $N$ in such a way that
the image of a small tubular neighborhood of $S_0(f_4)$
in $M$ is disjoint from $\Int{\Delta}$.
This is possible, since $f_3|_{S_0(f_3)}$ and such a curve
are homotopic to each other.
Then, by further modifying $f_4$, we may assume that
$N(\Delta) \cap f_4(S(f_4)) = f_4(S_0(f_4))$,
where $N(\Delta)$
is a small neighborhood of $\Delta$ in $N$.

Finally, we can use a set of moves
as explained in \cite[Fig.~7]{GK0}: we apply two flips, which create
two pairs of cusp points, a Reidemeister II type move applied to 
the definite fold image that is realized by a disk move, 
a Reidemeister II type move applied to the indefinite fold image,
which is seen to be realizable by an index argument,
and then unflips (see \cite[Lemma~4.8]{GK}), which eliminate cusps
and definite fold points.

This completes the proof.
\end{proof}

\begin{rmk}
In \cite{Saeki06}, regular homotopy
was decomposed into Reidemeister type moves II and III, and 
each such move was realized by a homotopy of maps.
However, 
strictly speaking, it is not enough:
one needs to use isotopies as well. Usually, this causes
no problem: however, in our situation, it does, since
the image of the definite fold component
may intersect with the image of the indefinite fold
components. Unfortunately, this was not thoroughly explained in \cite{Saeki06}.
\end{rmk}

\begin{rmk}\label{rmk:tau}
In the course of the proof in \cite{Saeki06}, we have used the connected sum
operations with the map $\tau : S^n \to \R^2$, constructed 
in \cite[Example~2.2]{Saeki06}. 
In fact, this is also realized as a composition of the homotopy moves as
depicted in Figure~\ref{fig12}. 
In the figure, the dotted curve represents the image of the
joining curve used for the cusp merge, and the image
of the definite (or indefinite) fold is depicted by thick 
(resp.\ thin) curves.
\end{rmk}

\begin{figure}[htbp]
\begin{center}
\psfrag{0}{$0$}
\psfrag{1}{$1$}
\psfrag{f}{Flip}
\psfrag{c}{Cusp Merge}
\includegraphics[width=\linewidth,height=0.4\textheight,
keepaspectratio]{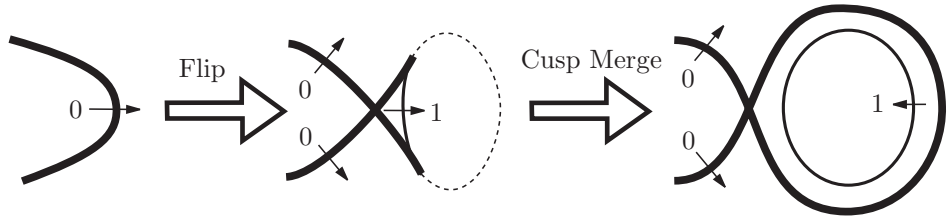}
\end{center}
\caption{Moves that realize the connected sum with $\tau$. 
In this and the following figures, thick lines depict definite fold images,
while thin lines depict indefinite fold images. The labels attached
to small arrows indicate the index in the designated direction:
e.g.\ label $0$ corresponds to a definite fold image and
the fiber over the region the arrow points into has one
additional $(n-2)$--sphere component when compared with the fiber
over the region the arrow starts from.}
\label{fig12}
\end{figure}

\begin{rmk}
Note that our new proof has the advantage that
the modifications are performed in the same homotopy class.
Furthermore, the procedures can be realized algorithmically:
we start with the procedures for making
the definite fold a circle, then use disk moves 
and the moves described in Remark~\ref{rmk:tau} to get
embedded definite fold image, and finally use the
procedure described in \cite[Fig.~7]{GK0}.

Note that such an algorithmic construction is an important ingredient
in \cite{BS, BS2} for constructing simplified broken
Lefschetz fibrations and simplified trisections on
$4$--dimensional manifolds. 
\end{rmk}

\section{Examples}\label{section3}

As has been pointed out, our new proof is constructive,
which enables us to give explicit examples as follows.
Let us start with examples in dimension three.
 
\begin{exam}\label{connected}
Let us consider the positive Hopf fibration $S^3 \to S^2$.
This is a non-singular map and hence is a stable map. By applying the
birth (see \cite[Lemma~3.1 and Rmark~3.2]{Saeki3})
and then a cusp merge as depicted in Figure~\ref{fig14} (1), we get
a stable map with one definite fold circle whose image is embedded. Then,
we can apply the moves as in \cite[Fig.~7]{GK0} to get a stable map
without definite fold. Note that the resulting stable map
has two indefinite fold circles that
are disjointly embedded into $S^2$.

For a positive integer $n$, by taking the connected sum of 
$n$ copies of the above stable map, we get a stable map
$S^3 \to S^2$ representing $n \in \Z \cong \pi_3(S^2)$ as follows.
For $n=2$, let $f_1$ and $f_2 : S^3 \to S^2$ be
stable maps each of which has one definite fold circle whose image is embedded
as in the right picture of Figure~\ref{fig14} (1). By
composing diffeomorphisms isotopic to the identity to $f_1$ and $f_2$
if necessary, we may assume that $f_1(S(f_1))$ (resp.\ $f_2(S(f_2))$) lies
in the northern (resp.\ southern) hemisphere of $S^2$
in such a way that they form circles parallel to the equator
of $S^2$ and that the images of the definite fold circles are
both adjacent to the equator.
Let $\Delta_1$ (resp.\ $\Delta_2$)
be the $2$--disk bounded by $f_1(S_0(f_1))$ (resp.\  $f_2(S_0(f_2))$)
lying in the northern (resp.\ southern) hemisphere. Then, 
$\Delta_j$ contains the image of a small regular neighborhood
of the definite fold, $j = 1, 2$, and we may assume that
$f_1(p_1)$ and $f_2(p_2)$
are close to each other for some definite fold points $p_j \in S_0(f_j)$, $j = 1, 2$.
Let $D_j$ be a small closed $2$--disk neighborhood of $f_j(p_j)$
in $S^2$, $j = 1, 2$,
such that $D_j \cap f_j(S(f_j))$ is diffeomorphic to a line segment consisting
of definite fold images and that $\partial D_j$ intersects $f_j(S(f_j))$
transversely. Then, $f_j^{-1}(D_j)$ contains a connected
component $B_j$ diffeomorphic to the $3$--ball that contains $p_j$ in its
interior. Now, we can construct a stable map $f_1 \sharp f_2$
into $S^2$ by gluing $f_j$ restricted to $S^3 \setminus \Int{B_j}$, $j = 1, 2$,
in an appropriate way.
(As to the gluing procedure, refer to
\cite{Saeki2}. See also Figure~\ref{fig14} (2).) 
For a general positive integer $n$, we repeat the same procedure.
Note that in the gluing process, we need to make sure that
the orientations of the source manifolds are consistent so that
we get a map representing $n \in \Z \cong \pi_3(S^2)$.

\begin{figure}[htbp]
\centering
\psfrag{m}{Cusp Merge}
\psfrag{c}{Connected Sum}
\psfrag{a}{(1)}
\psfrag{2}{(2)}
\psfrag{0}{$0$}
\psfrag{1}{$1$}
\includegraphics[width=\linewidth,height=0.6\textheight,
keepaspectratio]{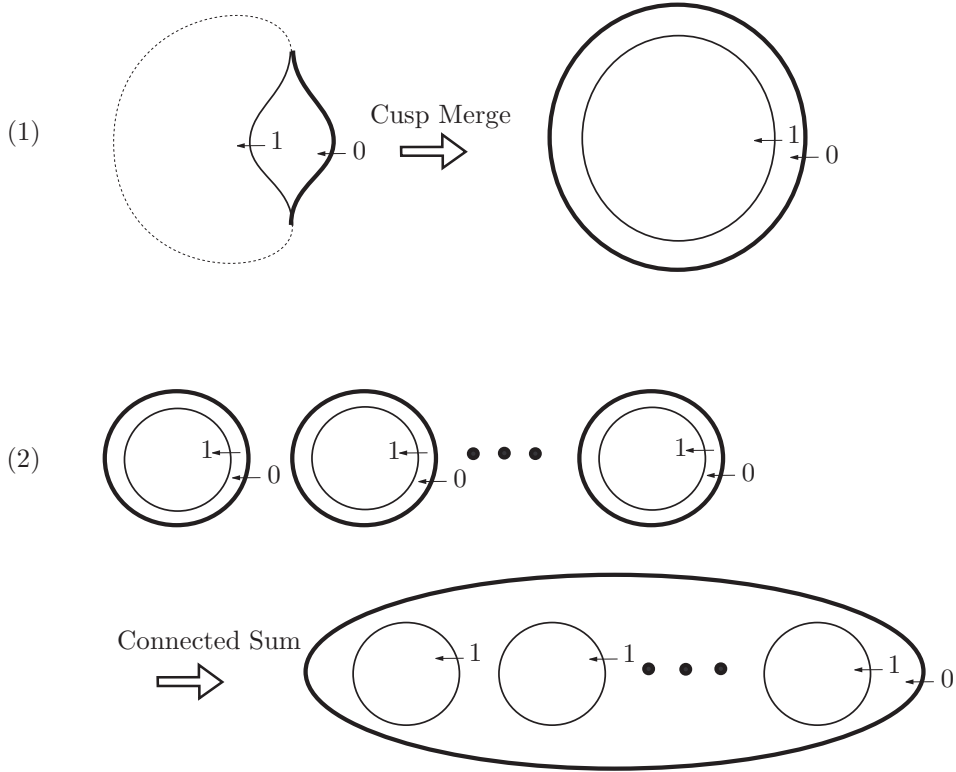}
\caption{(1) A stable map $S^3 \to S^2$ homotopic to the positive Hopf fibration.
(2) A stable map $S^3 \to S^2$ corresponding to a positive
integer $n \in \Z \cong \pi_3(S^2)$, where
we take the connected sum of $n$ copies of the stable map homotopic to the
positive Hopf fibration. Thick (or thin)
curves indicate the image of definite (resp.\ indefinite)
fold points.}
\label{fig14}
\end{figure}

Then, by applying the moves as in \cite[Fig.~7]{GK0},
we get a stable map
without definite fold. Note that the resulting stable map
has $n+1$ indefinite fold circles that are disjointly embedded.

For non-positive integers $n$, we can also use
stable maps that are homotopic to the negative Hopf fibration.
In this way, for every homotopy class $S^3 \to S^2$, we get
an explicit stable map without definite fold.

Note that in the gluing process above, if we use $k$ positive Hopf
fibrations and $\ell$ negative ones for some non-negative integers
$k$ and $\ell$ with $k + \ell > 0$, then their connected sum yields a stable map
without definite fold such that it has $k+\ell+1$ indefinite fold circles
that are disjointly embedded into $S^2$ and which represents
$k-\ell \in \Z \cong \pi_3(S^2)$. This observation shows that
there exist pairs of stable maps $S^3 \to S^2$ without definite fold which have exactly
the same indefinite fold images but which are not homotopic to each other.
\end{exam}

\begin{prob}\label{problem}
For an integer $n \in  \Z \cong \pi_3(S^2)$, let us
consider stable maps $f : S^3 \to S^2$ without definite fold which represent
the associated homotopy class and which satisfies that $S(f) \neq \emptyset$ and
$f|_{S(f)}$ is an embedding. Then, is the number of components of
$S(f)$ congruent modulo two to $n+1$? Furthermore, is the
minimum number of components of $S(f)$ over all such $f$
equal to $|n|+1$?\footnote{This problem
is originally due to Masamichi Takase.}
\end{prob}

Now, let us consider the case of dimension four.

\begin{exam}\label{ex2.6}
Some examples of stable maps without definite fold points
are given in \cite[\S8.2]{ADK}. The map $S^4 \to S^2$ given
there corresponds, in fact, to the one obtained
in our simple proof starting from the standard stable map
$S^4 \to \R^2 \to S^2$ with only definite fold points as its singularities,
where the first map is the standard projection $\R^5 \to \R^2$
restricted to the unit $4$--sphere and the second map is an embedding.
\end{exam}

\begin{exam}
Matsumoto \cite{Matsumoto} constructed
a Lefschetz fibration $S^4 \to S^2$ of genus $1$ with a single ``twin'' singular fiber,
as the composition $h \circ \Sigma h$, where $\Sigma h : S^4 = \Sigma S^3 \to \Sigma S^2 = S^3$
is the suspension of the Hopf fibration $h: S^3 \to S^2$.
As is well-known, this represents the non-trivial element of $\pi_4(S^2) \cong \Z_2$.
Note that the twin singular fiber has two Lefschetz critical points.
It is known that, by deforming the fibration slightly, the twin singular fiber
splits into two singular fibers (of type $I_1^\pm$ in the notation in \cite{Matsumoto2}).
By the wrinkling moves (see, for example, \cite{L}), we can
transform the two Lefschetz critical points to two circles
of indefinite fold and cusp points each of which contains exactly three cusp points.
Then, by applying the cusp merges three times, we get
three circles of indefinite fold embedded by the projection to $S^2$, 
where the image of one circle
component encircles the other two. The result is an example
of a stable map without definite fold whose homotopy class
represents the non-trivial element of $\pi_4(S^2)$.
Note that the stable map $S^4 \to S^2$ constructed
in Example~\ref{ex2.6} represents the neutral element of
$\pi_4(S^2)$.
\end{exam}

\begin{exam}
As has been shown in \cite{Saeki2}, for every homotopy $n$--sphere
$\Sigma^n$ with $n \geq 5$, which may possibly be an exotic
$n$--sphere \cite{KM},
there exists a stable map $g : \Sigma^n \to \R^2$
with only definite fold as its singularities. In fact, the image is
diffeomorphic to the $2$--disk. By embedding the $2$--disk into $S^2$,
and by applying the procedure as depicted in \cite[Fig.~7]{GK0}, we
get a stable map $f : \Sigma^n \to S^2$ with only one circle of indefinite fold
points.
\end{exam}

In the above examples, the map $f$ restricted
to $S(f)$ is an embedding. If we impose this condition
for stable maps without definite fold, then we do not
know if the number of components of $S(f) = S_1(f)$
is minimal in the given homotopy class, unless it
is connected (see also Problem~\ref{problem}). 
Note that if we allow $f|_{S(f)}$ to
have self-intersections, then we can always
arrange so that $S(f_1) = S_1(f_1)$ is connected
by using the techniques developed in \cite{Saeki3}.

\begin{rmk}
As has been shown in \cite[Proposition~2.11]{Saeki06}, for maps $f : M \to N$
into a general surface $N$, a similar result does not hold. In fact,
the finiteness of the index $[\pi_1(N); f_*\pi_1(M)]$ is a necessary
condition for the existence of a stable map without definite fold
homotopic to $f$. Later, Gay and Kirby \cite{GK} showed that this
is, in fact, sufficient, answering to the author's question
posed in \cite{Saeki06}.
\end{rmk}

\begin{rmk}
After \cite[Theorem~2.6]{Saeki06} (or Theorem~\ref{thm}
of the present paper) was proved, it was used for proving
the existence of a broken Lefschetz fibration
on an arbitrary closed orientable $4$--manifold \cite{Baykur}.
(In \cite{Baykur} Baykur also gave a topological proof of the existence of
broken Lefschetz pencils on near-symplectic $4$--manifolds, using the theorem.)
Furthermore, Gay and Kirby \cite{GK} extensively generalized
the theorem in more general settings.
For example, they showed that for a closed connected
manifold of dimension $n > 2$ and a closed connected
surface $N$, a continuous map $f : M
\to N$ is homotopic to a stable
map without definite fold points if and only if
the index $[\pi_1(N): f_*\pi_1(M)]$ is finite.
\end{rmk}

\section{Simple stable maps without definite fold}\label{section4}

In \cite{Saeki4}, simple stable maps of closed orientable
$3$--manifolds into $\R^2$ or $S^2$ were studied. In this section,
combining the techniques developed there with those in this paper,
we study simple stable maps without definite fold. 
Recall that a stable map is \emph{simple}
if it does not have cusp points
and every component of the inverse image of a point in the target always contains
at most one singular point. For example, if a stable map
$f$ satisfies that $f|_{S(f)}$ is an embedding, then $f$ is necessarily simple.

We first prove the following.

\begin{prop}\label{prop1}
Let $M$ be a closed orientable $3$--manifold. Then, the
following conditions are equivalent.
\begin{itemize}
\item[(i)] $M$ is a graph manifold, i.e.\ it is the union of
finitely many $S^1$--bundles over compact surfaces attached along their
torus boundaries.
\item[(ii)] There exists a stable map $f : M \to S^2$ without definite fold
such that $f|_{S(f)}$ is an embedding.
\item[(iii)] There exists a simple stable map $f : M \to S^2$ without definite
fold.
\end{itemize}
\end{prop}

\begin{proof}
If $M$ is a graph manifold, then by \cite{Saeki4}, there exists
a stable map $g : M \to S^2$ without cusps such that $g|_{S(g)}$ is an embedding.
Then, by using the method described in the last step of proof of Theorem~\ref{thm},
we can replace each definite fold circle with an indefinite fold circle by homotopy.
Furthermore, we can arrange so that the restriction to the singular point
set of the resulting map is still an embedding.
Thus, (i) implies (ii). If (ii) is satisfied, then the same stable map satisfies the
condition in (iii). If (iii) is satisfied, then (i) holds, as is proved in \cite{Saeki4}.
This completes the proof.
\end{proof}

Let $M$ be a closed oriented graph $3$--manifold.
Yano \cite{Yano} introduced a subgroup $G(M)$ of $H_1(M; \Z)$ and
proved that a homology class $\alpha$ of $H_1(M; \Z)$ can be represented
by a graph link if and only if $\alpha \in G(M)$. Here, a link $L$
in $M$ is a \emph{graph link} if its exterior $M \setminus
\Int{N(L)}$ is a graph manifold, where $N(L)$ is a tubular neighborhood of $L$ in $M$.
On the other hand, let $[M, S^2]$ be the set of homotopy classes
of continuous maps $M \to S^2$, and let 
$\mathrm{deg} : [M, S^2] \to H_1(M; \Z)$
be the map which associates to an element represented by a smooth
map $M \to S^2$ the homology class represented by a regular fiber
(see \cite{CRS}). Here, we fix an orientation of $S^2$, and
each regular fiber is oriented in accordance with the orientations
of $M$ and $S^2$. Note that it is known that $\mathrm{deg}$
is a well-defined surjective map, which can be seen by the Pontrjagin--Thom construction.
Then, for graph $3$--manifolds, we have the following.

\begin{prop}\label{graph2}
Let $M$ be a closed oriented graph $3$--manifold.
A continuous map $g : M \to S^2$ is homotopic to a stable
map $f : M \to S^2$ without definite fold 
such that $f|_{S(f)}$ is an embedding
if and only if the homotopy class of $g$ lies in $\mathrm{deg}^{-1}(G(M))$.
\end{prop}

In particular, if the Jaco--Shalen--Johannson complex of $M$
is a tree (see \cite{Yano}), then every continuous map
$M \to S^2$ is homotopic to a stable
map $f : M \to S^2$ without definite fold such that $f|_{S(f)}$ is an embedding.
For example, if $H_1(M; \Z)$ is finite, then this always holds.

\begin{proof}[Proof of Proposition~\textup{\ref{graph2}}]
Let $f : M \to S^2$ be a stable map without definite fold
such that $f|_{S(f)}$ is an embedding. Then, by \cite{Saeki4},
every regular fiber is a graph link. Hence, the homotopy class
of $f$ must necessarily lie in $\mathrm{deg}^{-1}(G(M))$.

Conversely, let $g : M \to S^2$ be a continuous map
whose homotopy class lies in $\mathrm{deg}^{-1}(G(M))$.
By \cite{Yano},
there exists an oriented graph link $L$ in $M$ that represents a
homology class associated with the homotopy class of $g$. 
Then, there is a decomposition of $M$ into a finite number of $S^1$--bundles
over compact surfaces attached along their
torus boundaries such that each component of $L$ is an $S^1$--fiber
of some bundle piece. Then, we mostly follow the procedures
described in \cite[\S 4]{Saeki4} as follows. We first decompose the $S^1$--bundle
pieces in such a way that each base surface has genus $0$.
Then, we embed each base surface into $S^2$. At this stage,
in \cite{Saeki4}, we embedded the surfaces so that their
images were disjoint. However, in our present situation, we
embed the base surfaces so that the union of all the boundary
curves are embedded, that the base points over which lie
a component of $L$ are mapped to the same point, say $x \in S^2$, and
that no other points are mapped to $x$.
Furthermore, we arrange the embeddings so that the $S^1$--fibers
have the correct orientations. Then, we follow the procedures
as described in \cite[\S 4]{Saeki4} to get a stable map $f_1 : M \to S^2$
such that $f_1|_{S(f_1)}$ is an embedding and that $(f_1)^{-1}(x)$ contains
$L$ as an oriented link.

Note that $f_1$ may have definite fold points. However, each definite fold
circle is embedded: therefore, by the final procedure as described
in the proof of Theorem~\ref{thm}, we can replace each definite fold
component with an indefinite fold component, one by one by homotopy. 
Thus, we get a stable map
$f_2 : M \to S^2$ without definite fold such that 
$f_2|_{S(f_2)}$ is an embedding and that $(f_2)^{-1}(x)$ contains
$L$ as an oriented link.

Then, by performing certain surgery operations to $f_2$
on a neighborhood of $(f_2)^{-1}(x) \setminus L$
as described in \cite[Proof of Lemma~3.6]{Saeki4},
we get a stable map $f_3 : M \to S^2$ such that $(f_3)^{-1}(x)$
coincides with $L$. However, during the surgery operation,
we need to create definite fold and $S_0(f_3)$ is not empty in
general. As each component of $S_0(f_3)$ is embedded by $f_3$,
we can eliminate each such component by homotopy
to get a stable map $f_4 : M \to S^2$ without definite fold.
After each such homotopy, we can observe that a circle
component is added to the inverse image of the point $x$.
However, each such component is a fiber situated
near a definite fold, and we see that the union of all
such components constitute a trivial link. In particular,
$(f_4)^{-1}(x)$ is an oriented link which is $\Z$--homologous
to $L$.

Now, the homotopy classes of $f_4$ and $g$ have the same
image by $\mathrm{deg}$. Then, according to \cite{CRS},
we see that by taking the connected sum of $f_4$ with
finitely many positive (or negative) Hopf fibration maps $S^3 \to S^2$,
we can arrange so that the resulting map is homotopic to $g$.
On the other hand, such connected sum operations can be performed
as described in Example~\ref{connected}: we first create
a definite fold circle for each of the maps, we take connected sum
along definite fold circles, and then we eliminate the definite fold.
In this way, we get a stable map $f_5 : M \to S^2$ without
definite fold homotopic to $g$ such that $f_5|_{S(f_5)}$
is an embedding.
This completes the proof.
\end{proof}

\section{Non-existence of singular Legendre fibrations}\label{section5}

Let $M$ be a closed orientable $3$--manifold endowed with a contact structure
and $f : M \to N$ a stable map
of $M$ into a surface $N$ without definite fold nor cusp points.
Note that then, for every $y \in N$, the fiber $f^{-1}(y)$
is a union of immersed circles in $M$.
Such a map $f$ is called a
\emph{singular Legendre fibration}
if each fiber $f^{-1}(y)$, $y \in N$,
is a union of Legendre curves: i.e.\ if each
fiber is tangent to the plane field 
(or the $2$--dimensional distribution)
given by the contact structure.
In \cite{Saeki06} we proposed the following problem, which
was originally due to Goo Ishikawa:

\begin{prob}\label{prob3.4}
Determine those $C^\infty$ stable
maps $f : M \to N$ which are singular Legendre fibrations
for some contact structure on $M$.
\end{prob}

In this section, we solve the above problem negatively.

\begin{prop}\label{prop:l}
If a $C^\infty$ stable map $f : M \to N$ has non-empty
singular point set, then it cannot be a singular Legendre fibration.
\end{prop}

\begin{proof}
Suppose $f$ is a singular Legendre fibration with respect
to a contact structure $\xi$ on $M$. Let $p \in S(f)$ be
a singular point, which is necessarily an indefinite fold point as
has been pointed out in \cite{Saeki06}.
There exist local coordinates $(x, y, z)$ around $p$, and
$(X, Y)$ around $f(p)$ such that $f$ has the form
$$X \circ f(x, y, z) = x^2 - y^2, \, Y \circ f(x, y, z) = z.$$
Suppose that the contact structure $\xi$ is locally given
by a non-degenerate $1$--form $\alpha$ of the form
$$\alpha = \varphi_1 dx + \varphi_2 dy + \varphi_3 dz$$
for some $C^\infty$ functions $\varphi_i$, $i = 1, 2, 3$, defined
locally around $p$.

First note that the fiber over $f(p)$ contains the set
$$\{(x, y, z)\,|\, x^2-y^2=0, z = 0\},$$
and hence the tangent vectors for the two crossing curve
segments at $p$ span the vector space defined by $dz = 0$
at $p = 0$. This implies that 
$$\varphi_1(0) = \varphi_2(0) = 0, \, \varphi_3(0) \neq 0.$$
Thus, at $p$, we have
$$\alpha \wedge d\alpha|_p = \varphi_3(0)\left(\frac{\partial \varphi_2}{\partial x}(0)
- \frac{\partial \varphi_1}{\partial y}(0)\right)dx \wedge dy \wedge dz|_p.$$
Since $\alpha$ is non-degenerate, we have
$$\frac{\partial \varphi_2}{\partial x}(0)
- \frac{\partial \varphi_1}{\partial y}(0) \neq 0.$$

On the other hand, the local vector field $v$
defined around $p$ by $v = (y, x, 0)$ is tangent
to the fibers. Therefore, $v$ should lie
on the planes defined by $\alpha = 0$.
This implies that 
$y \varphi_1 + x \varphi_2$ constantly vanishes.
Since $\varphi_1(0) = \varphi_2(0) = 0$, by the Hadamard lemma,
we have
$$\varphi_1 = x g_1 + y h_1 + zk_1, \, 
\varphi_2 = x g_2 + y h_2 + zk_2$$
for some $C^\infty$ functions $g_1, h_1, k_1, g_2, h_2$ and $k_2$ defined
near $p$. In this case, we have
$$\frac{\partial \varphi_1}{\partial y}(0) = h_1(0),\,
\frac{\partial \varphi_2}{\partial x}(0) = g_2(0).$$
As 
$$y \varphi_1 + x \varphi_2 = xy g_1 + y^2 h_1 + yz k_1
+ x^2 g_2 + xy h_2 + xz k_2$$
constantly vanishes, by differentiating the above function
with respect to $y$ twice and substituting $(x, y, z) = (0, 0, 0)$,
we see that $h_1(0) = 0$. Similarly, we have $g_2(0) = 0$.
Therefore, we have
$$\frac{\partial \varphi_2}{\partial x}(0)
- \frac{\partial \varphi_1}{\partial y}(0)
= g_2(0) - h_1(0) = 0,$$
which is a contradiction.
This completes the proof.
\end{proof}

The above proof shows, in fact, that there is no
contact structure which makes local fibers near
an indefinite fold all Legendrian.

Recall that if an empty singular point set is allowed, then
there do exist (non-singular) Legendre fibrations
(see \cite{Giroux}, \cite[Proposition~1.1.7]{ET}).

\section*{Acknowledgment}\label{ack}
The author would like to express his hearty thanks to
\.{I}nan\c{c} Baykur, whose comments drastically improved
the presentation of the article. The author also would like to thank
Rustam Sadykov and Masamichi Takase for
stimulating discussions and invaluable comments and questions.
The author has been supported in part by JSPS KAKENHI Grant Numbers 
JP23244008, JP23654028, JP15K13438, JP16K13754, JP16H03936, JP17H01090,
JP17H06128.

\end{document}